\documentclass[a4paper]{article}

\usepackage[leqno,intlimits]{amsmath}
\usepackage{amssymb}
\usepackage{amsthm}
\usepackage{hyperref}
\usepackage{bbm}

\usepackage{pstricks}
\usepackage{multido}
\usepackage{calc}

\psset{unit=1pt}
\psset{linewidth=0.5pt}

\newtheorem{tm}{Theorem}[section]
\newtheorem{lm}[tm]{Lemma}
\newtheorem{rem}[tm]{Remark}
\newtheorem{cor}[tm]{Corollary}

\numberwithin{equation}{section}

\newcommand*{\omin}{\omega^\text{min}}
\newcommand*{\omax}{\omega^\text{max}}
\newcommand*{\emin}{\eta^\text{min}}
\newcommand*{\emax}{\eta^\text{max}}
\newcommand*{\il}{\ell}
\newcommand*{\ir}{\mathfrak r}

\newcommand*{\Lc}{\mathcal L}
\newcommand*{\Uc}{\mathcal U}
\newcommand*{\pf}{\mathfrak p}
\newcommand*{\qf}{\mathfrak q}
\newcommand*{\Om}{\Omega}
\newcommand*{\om}{\omega}
\newcommand*{\de}{\delta}
\newcommand*{\La}{\Lambda}
\newcommand*{\Zb}{\mathbb Z}
\newcommand*{\Rb}{\mathbb R}
\newcommand*{\un}{\underline}
\newcommand*{\vp}{\varphi}
\newcommand*{\ve}{\varepsilon}
\newcommand*{\vr}{\varrho}
\newcommand*{\te}{\theta}
\newcommand*{\tmin}{\theta^\text{min}}
\newcommand*{\tmax}{\theta^\text{max}}
\newcommand*{\e}[1]{\text{\rm e}^{#1}}
\newcommand*{\al}{\alpha}

\begin{document}

\title{Product blocking measures and a particle system proof of the Jacobi triple product}

\author{
 M\'arton Bal\'azs
 \thanks{University of Bristol; \texttt{m.balazs@bristol.ac.uk}; research partially supported by the Hungarian Scientific Research Fund (OTKA/NKFIH) grants K100473 and K109684.}
 \and
 Ross Bowen
 \thanks{University of Bristol}
}

\date{\today}

\maketitle

\begin{abstract}
 We review product form blocking measures in the general framework of nearest neighbor asymmetric one dimensional misanthrope processes. This class includes exclusion, zero range, bricklayers, and many other models. We characterize the cases when such measures exist in infinite volume, and when finite boundaries need to be added. By looking at inter-particle distances, we extend the construction to some 0-1 valued particle systems e.g., $q$-ASEP and the Katz-Lebowitz-Spohn process, even outside the misanthrope class. Along the way we provide a full ergodic decomposition of the product blocking measure into components that are characterized by a non-trivial conserved quantity. Substituting in simple exclusion and zero range has an interesting consequence: a purely probabilistic proof of the Jacobi triple product, a famous identity that mostly occurs in number theory and the combinatorics of partitions. Surprisingly, here it follows very naturally from the exclusion -- zero range correspondence.
\end{abstract}

\small{

}

\section{Introduction}

Stationary distributions in one dimensional asymmetric simple exclusion (ASEP) have been fully explored by Liggett \cite{couse}. The extremal ones are translation-invariant iid.\ Bernoulli distributions characterized by a density \(0\le\vr\le1\), and non translation-invariant distributions that come from a product but not identically distributed Ber\-no\-ul\-li measure by conditioning on a conserved quantity of the dynamics. The former are non-reversible, while the latter are. This manuscript focuses on such non-homogeneous reversible product stationary structures in several models. Such measures are mentioned sporadically in the literature. A common feature of these distributions is that, looking in the direction of the asymmetry, the particle numbers go from the least allowed per site to the most possible, hence the name blocking measure.

The aim of this manuscript is two-fold: first, to show that product blocking measures occur very generally, and describe them in a misanthrope-type framework that includes many of the known examples like ASEP, \(K\)-exclusion, zero range (ZRP). It is remarkable, and has been known e.g., for \(K\)-exclusion, that product blocking measures even exist in models where the translation-invariant distributions are not of product form. As it turns out the blocking scenario allows us to keep our state space countable for the cases we are interested in, hence construction of dynamics comes cheap using attractivity. We cite known results in particular cases of models in Section \ref{sc:cover}. Our treatment in the unifying framework of misanthrope processes adds some novelty compared to the literature. However, it paves the way to the next observation which is genuinely new and interesting.

Second, we demonstrate that blocking measures have a rather rich algebraic structure. In particular, we give a very natural (in the particle systems context) proof of the \emph{Jacobi triple product}, a central formula in several branches of mathematics which seemingly had little to do with probability so far. Along the way we only use the most classical bits of the field of interacting particle systems, namely ASEP, constant rate zero range, their well-known reversible blocking measures, and a much used transformation that takes one of these models into the other. The job becomes interesting due to the fact that the blocking scenario of ASEP takes place on the whole of \(\Zb\) with a conserved quantity, while for ZRP one restricts to the half-line \(\Zb^-\) with an appropriately chosen boundary reservoir at site 1. The conserved quantity of ASEP then requires us to do a full ergodic decomposition of the product measure that arises from the ZRP as it is transformed into ASEP. The fact that a probabilistic proof of a rather involved identity follows from manipulations of the most elementary objects in particle systems came as a pleasant surprise.

We introduce our framework in Section \ref{sc:mod}, then state and explain our results in Section \ref{sc:res}. Several examples that directly fall in the family under consideration are briefly mentioned in Section \ref{sc:cover}. We then turn to describing the well-known transformation of certain processes into others which are usually outside the misanthrope class we started with. This is done in Section \ref{sc:standup} together with examples of models to which our results extend this way. One notable exception is ZRP, which the transformation maps into ASEP, another model we fully cover. This will lead to the proof of the Jacobi triple product in Section \ref{sc:jacobi}, following some general proofs in Section \ref{sc:gen}.

\section{The models}\label{sc:mod}

The family of models we investigate is closely related to the \emph{misanthrope processes} introduced by Cocozza-Thivent \cite{coco} and further generalized in T\'oth and Valk\'o \cite{hydro} and in Bal\'azs \cite{fluct}. The deviation form their original setup is that we shall sometimes restrict the volume from \(\Zb\) to a half-infinite integer line or to a finite segment, and that in search for reversible product distributions we can relax some assumptions that were required before to obtain certain telescopic properties of the generator.

Given are two, possibly infinite, integers \(-\infty\le\il\le0\le\ir\le\infty\), and our dynamics will take place over the integer one-dimensional lattice \(\La:\,=\{i\,:\,\il-1<i<\ir+1\}\subseteq\Zb\). The definition of the model also involves two other, possibly infinite, integers \(-\infty\le\omin\le0<\omax\le\infty\), and we define \(I:\,=\{z\,:\,\omin-1<z<\omax+1\}\subseteq\Zb\). The state space \(\Om\) of our process will be a subset of \(I^\La\). Special restrictions will be placed to ensure that \(\Om\) stays countable.

To proceed with the definition we introduce the \(I^\La\to\Zb^+\cup\{\infty\}\) functions
\begin{equation}
 N_\text p(\un z):\,=\sum_{i=\il}^0(z_i-\omin)\qquad\text{and}\qquad N_\text h(\un z):\,=\sum_{i=1}^\ir(\omax-z_i).\label{eq:NpNh}
\end{equation}
Notice that when \(\omin=0\) and \(\omax=1\), as is the case for the asymmetric simple exclusion (ASEP) below, \(N_\text p\) counts the number of particles (\(z_i=1\)) on the left of position \(\frac12\), and \(N_\text h\) the number of holes (\(z_i=0\)) right of position \(\frac12\). On this intuition we call \(z_i-\omin\) the number of particles (and \(\omax-z_i\) the number of holes) at site \(i\) whenever \(\omin\) (\(\omax\), respectively) is finite. Our state space is defined as
\begin{equation}
 \Om:\,=\bigl\{\un z\in I^\La\,:\,\bigl(\il>-\infty\text{ or }N_\text p(\un z)<\infty\bigr)\text{ and }\bigl(\ir<\infty\text{ or }N_\text h(\un z)<\infty\bigr)\bigr\}.\label{eq:Omdef}
\end{equation}
In words, if the volume \(\La\) is infinite to the left (\(\il=-\infty\)) then we restrict \(\omin>-\infty\) and the state space can only have finitely many particles on the left of position \(\frac12\). The interpretation is similar with restricting \(\omax\) and the number of holes on the right of \(\frac12\) when \(\ir=\infty\). Even when \(\il=-\infty\) (\(\ir=\infty\)), partitioning w.r.t.\ finite \(\min\{i\,:\,z_i\ne\omin\}\) and \(\max_{i\le0}(z_i-\omin)\) (\(\max\{i\,:\,z_i\ne\omax\}\) and \(\max_{i>0}(\omax-z_i)\), respectively), we see that \(\Om\) is indeed countable.

We define our Markovian evolution in one, two or three pieces, depending on finiteness of \(\il\) and \(\ir\). Define the possible jumps
\[
 \bigl({\un z}^{i,j}\bigr)_k:\,=\left\{
  \begin{aligned}
   &z_k,&&\text{if }k\ne i,\,j,\\
   &z_i-1,&&\text{if }k=i,\\
   &z_j+1,&&\text{if }k=j
  \end{aligned}
 \right.
\]
for \(k\in\La\). Notice that this definition makes sense even if \(i\) or \(j\) is not in \(\La\), we will use this at the boundaries of \(\La\). Fix \(\pf\) and \(\qf\,:\,I^2\to\Rb^+\) rate functions with
\begin{equation}
 \pf(\omin,\,\cdot)\equiv\pf(\cdot,\,\omax)\equiv\qf(\omax,\,\cdot)\equiv\qf(\cdot,\,\omin)\equiv0\label{eq:bulkb}
\end{equation}
whenever \(\omin\) (respectively, \(\omax\)) is finite; further restrictions on these rates will apply. The bulk part of the dynamics takes place inside \(\La\), and is governed by the generator
\[
 (L^\text{bulk}\vp)(\un z):\,=\sum_{i=\il}^{\ir-1}\bigl[\pf(z_i,\,z_{i+1})\cdot\bigl(\vp({\un z}^{i,i+1})-\vp(\un z)\bigr)+\qf(z_i,\,z_{i+1})\cdot\bigl(\vp({\un z}^{i+1,i})-\vp(\un z)\bigr)\bigr].
\]
Notice that the sum is finite for any \(\un z\in\Om\), hence no restriction is needed on the function \(\vp\,:\,\Om\to\Rb\).

When \(\il>-\infty\), we define the left boundary rates \(\pf_\il\) and \(\qf_\il\,:\,I\to\Rb^+\) with
\begin{equation}
 \qf_\il(\omin)=0,\qquad\text{and}\qquad\pf_\il(\omax)=0\text{ if }\omax<\infty,\label{eq:leftb}
\end{equation}
and the left boundary generator
\[
 (L^\il\vp)(\un z):\,=\pf_\il(z_\il)\cdot\bigl(\vp({\un z}^{\il-1,\il})-\vp(\un z)\bigr)+\qf_\il(z_\il)\cdot\bigl(\vp({\un z}^{\il,\il-1})-\vp(\un z)\bigr).
\]
When \(\il=-\infty\) we simply take \(L^\il\equiv0\).

Similarly, when \(\ir<\infty\), we define the right boundary rates \(\pf_\ir\) and \(\qf_\ir\,:\,I\to\Rb^+\) with
\begin{equation}
 \qf_\ir(\omax)=0,\qquad\text{and}\qquad\pf_\ir(\omin)=0\text{ if }\omin>-\infty,\label{eq:rightb}
\end{equation}
and the right boundary generator
\[
 (L^\ir\vp)(\un z):\,=\pf_\ir(z_\ir)\cdot\bigl(\vp({\un z}^{\ir,\ir+1})-\vp(\un z)\bigr)+\qf_\ir(z_\ir)\cdot\bigl(\vp({\un z}^{\ir+1,\ir})-\vp(\un z)\bigr).
\]
When \(\ir=\infty\) we just take \(L^\ir\equiv0\).

The boundary generators describe the interaction of the process with reservoirs at positions (if finite) \(\il-1\) and \(\ir+1\). The full dynamics of the process is given by the generator \(L=L^\text{bulk}+L^\il+L^\ir\).

Next we describe further restrictions on the jump rates. These closely follow \cite{hydro} and \cite{fluct}, \emph{except} for the order of triple sums condition and the symmetry of the function \(s\) further below, which they have but we do not need here. We assume attractivity of the process: whenever \(y,\,z,\,z+1\in I\),
\[
 \begin{aligned}
  \pf(z+1,\,y)&\ge\pf(z,\,y)\qquad&\qquad\qf(y,\,z+1)&\ge\qf(y,\,z),\\
  \pf(y,\,z+1)&\le\pf(y,\,z)\qquad&\qquad\qf(z+1,\,y)&\le\qf(z,\,y),\\
  \pf_\il(z+1)&\le\pf_\il(z)\qquad&\qquad\qf_\il(z+1)&\ge\qf_\il(z),\\
  \pf_\ir(z+1)&\ge\pf_\ir(z)\qquad&\qquad\qf_\ir(z+1)&\le\qf_\ir(z).\\
 \end{aligned}
\]
As the next lemma shows, these assumptions allow us to construct our processes in great generality, and will also be useful when formulating the stationary marginals later.
\begin{lm}\label{lm:noblowup}
 The dynamics with the assumptions so far is well defined and keeps the countable state space \(\Om\) for all times.
\end{lm}

Next we discuss irreducibility of the state space \(\Om\). To this order, the next assumption we make is
\begin{center}
 \emph{except for \eqref{eq:bulkb}, \eqref{eq:leftb} and \eqref{eq:rightb}, all jump rates \(\pf,\ \qf,\ \pf_\il,\ \qf_\il,\ \pf_\ir,\ \qf_\ir\) are non-zero.}
\end{center}
Naturally, our reversible dynamics can be cut anywhere by a freezing boundary condition, but we are not interested in these cases here, hence the assumption for the boundary rates.

If either \(\il\) or \(\ir\) is finite, then the boundary can add or remove an arbitrary number of particles and \(\Om\) is irreducible. An interesting phenomenon occurs, however, when both \(\il\) and \(\ir\) are infinite and we have no boundaries. Recall \eqref{eq:NpNh} and define the \(\Om\to\Zb\) function
\[
 N(\un z):\,=N_\text h(\un z)-N_\text p(\un z).
\]
As already observed by Liggett \cite{couse} for ASEP, this quantity is conserved by the dynamics. To see this notice that \(N\) could only change by a particle jumping from 0 to 1 or from 1 to 0. In both cases \(N_\text h\) and \(N_\text p\) change by the same amount (\(\mp1\)). Hence in this doubly infinite-volume case we set
\begin{equation}
 \Om^n:\,=\{\un z\in\Om\,:\,N(\un z)=n\}\label{eq:omndef}
\end{equation}
for \(n\in\Zb\), and notice that these are precisely the closed irreducible components for the dynamics on \(\Om\).

Finally, to prepare for product blocking measures, we impose further restrictions on the rates \(\pf\) and \(\qf\). We assume the existence of reals \(\frac12<p=1-q\le1\), a function \(f\,:\,I\to\Rb^+\) with \(f(\omin)=0\) if \(\omin\) is finite, and a function \(s\,:\,I\times I\to\Rb^+\) with which the bulk jump rates take the form
\[
 \pf(y,\,z)=p\cdot s(y,\,z+1)\cdot f(y)\qquad\text{and}\qquad\qf(y,\,z)=q\cdot s(y+1,\,z)\cdot f(z).
\]
In fact here we slightly extended the domain of \(s\) by \(s(\omax+1,\,\cdot)=s(\cdot,\,\omax+1)=0\) if \(\omax\) is finite. Attractivity implies that \(s\) is non-increasing in each of its variables, and \(f\) is non-decreasing. We also assume that \(f\) is such that the open interval \((\tmin,\,\tmax)\), to be defined below, is non-empty. E.g, \(f\) cannot be the constant function across the whole of \(I\): \(-\infty\le\inf_{z\in I}f(z)<\sup_{z\in I}f(z)\le\infty\) when \(\omin=-\infty\) and \(\omax=\infty\).

We give numerous examples of models in Section \ref{sc:cover}.

\section{Results}\label{sc:res}

Following Cocozza-Thivent \cite{coco}, T\'oth and Valk\'o \cite{hydro}, and Bal\'azs \cite{fluct}, we fix
\begin{equation}
 \begin{aligned}
  \tmin:&=\left\{
   \begin{aligned}
    &-\infty,&&\text{if }\omin>-\infty,\\
    &\lim_{z\to-\infty}\ln f(z),&&\text{if }\omin=-\infty,
   \end{aligned}
  \right.\\
  \tmax:&=\left\{
   \begin{aligned}
    &\infty,&&\text{if }\omax<\infty,\\
    &\lim_{z\to\infty}\ln f(z),&&\text{if }\omax=\infty.
   \end{aligned}
  \right.
 \end{aligned}\label{eq:tminmaxdef}
\end{equation}
For \(z\in I\) we abbreviate
\[
 f(z)!=\left\{
  \begin{aligned}
   &\prod_{y=1}^zf(z),&&\text{for }z>0,\\
   &1,&&\text{for }z=0,\\
   &\frac1{\prod_{y=z+1}^0f(z)},&&\text{for }z<0,
  \end{aligned}
 \right.\qquad\text{so that }f(z)!=f(z)\cdot f(z-1)!.
\]
Define, for \(\tmin<\te<\tmax\) the distribution
\begin{equation}
 \mu^\te(z):\,=\frac1{Z(\te)}\cdot\frac{\e{\te z}}{f(z)!},\label{eq:mutedef}
\end{equation}
with normalization
\[
 Z(\te)=\sum_{y\in I}\frac{\e{\te y}}{f(y)!}<\infty.
\]
Set \(\tmin<c<\tmax\), recall \(p>q\), and define
\begin{align}
 \te_i&=c+i\cdot(\ln p-\ln q),\qquad&&(i\in\Zb),\label{eq:teidef}\\
 \mu_i&=\mu^{\te_i},\qquad&&\text{if }\tmin<\te_i<\tmax.\notag
\end{align}
\begin{tm}\label{tm:mucstati}
 Suppose \(\La\) is such that \(\tmin<\te_i<\tmax\) for all \(i\in\La\). If \(\il>-\infty\), suppose \(\pf_\il\) and \(\qf_\il\) satisfy
 \begin{equation}
  \frac{\qf_\il(z+1)}{\pf_\il(z)}=\frac{f(z+1)}{\e{\te_\il}},\qquad\omax\ne z\in I.\label{eq:lcond}
 \end{equation}
 If \(\ir<\infty\), suppose \(\pf_\ir\) and \(\qf_\ir\) satisfy
 \begin{equation}
  \frac{\pf_\ir(z+1)}{\qf_\ir(z)}=\frac{f(z+1)}{\e{\te_\ir}},\qquad\omax\ne z\in I.\label{eq:rcond}
 \end{equation}
 Then the product distribution
 \begin{equation}
  \un\mu^c:\,=\bigotimes_{i\in\La}\mu_i\label{eq:mucdef}
 \end{equation}
 is reversible stationary for the process on the countable state space \(\Om\).
\end{tm}

\begin{rem}\label{rm:finite}
 When \(\tmax<\infty\), \eqref{eq:teidef} forces \(\ir<\infty\) in order for the process to have the above product stationary distribution. Similarly, \(\tmin>-\infty\) requires \(\il>-\infty\).
\end{rem}
In these cases a natural choice would be \(\pf_\ir(y)=\lim_{z\to\infty}\pf(y,\,z)\) and \(\qf_\ir(y)=\lim_{z\to\infty}\qf(y,\,z)\), and similarly \(\pf_\il(z)=\lim_{y\to-\infty}\pf(y,\,z)\) and \(\qf_\il(z)=\lim_{y\to-\infty}\qf(y,\,z)\). In the respective cases these limits always exist and are finite due to monotonicity of \(s\) and the respective limit conditions \eqref{eq:tminmaxdef}. However, they could be zero which we excluded for the boundary rates for irreducibility considerations. In this case this natural choice will not work. If the limits are non-zero, the assumptions \eqref{eq:lcond} and \eqref{eq:rcond} for these choices to work simply become that \(\tmax\) or \(\tmin\), respectively, become part of the arithmetic sequence \eqref{eq:teidef}. When both \(\il\) and \(\ir\) are finite, an arithmetic condition with increment \(\ln p-\ln q\) will decide whether this is simultaneously possible on both boundaries. If it all works out, then formally the boundaries can be thought of as infinitely many particles at site \(\ir+1\) and negative infinitely many particles at site \(\il-1\).

If either \(\il\) or \(\ir\) is finite, then the process is irreducible, and it follows from general Markov chain theory that the above distribution is the unique stationary distribution of the process (see e.g., Liggett \cite[Chapter 2.6]{liggett_cont_markov}). When both \(\il\) and \(\ir\) are infinite, the state space separates into the disjoint union of its irreducible components \(\Om^n\) \eqref{eq:omndef}. In this case, define the conditional distribution
\[
 \un\nu^n:\,=\un\mu^c(\cdot\,|\,N(\cdot)=n).
\]
\begin{lm}\label{lm:nunwd}
 When both \(\il\) and \(\ir\) are infinite, the distributions \(\un\nu^n\) are well defined for every \(n\in\Zb\) and \(\tmin<c<\tmax\), and do not depend on the value of \(c\).
\end{lm}
This is not very surprising as both \(c\) and \(n\) are directly related to shifts of configurations. To see this, change \(c\) by integer amounts of \(\ln p-\ln q\) in \eqref{eq:teidef}, and see \eqref{eq:nshift} later on for \(\un\nu^n\).

Notice that \(\un\nu^n\) arises by conditioning a reversible stationary distribution of a countable Markov chain on one of its irreducible components \(\Om^n\). The next statement therefore follows.
\begin{cor}\label{cr:nun}
 When both \(\il\) and \(\ir\) are infinite, the distribution \(\un\nu^n\) is the unique stationary distribution for the process on \(\Om^n\), and it is reversible.
\end{cor}

The proof of the above theorem will consist of a simple calculation. Formally, that calculation works out in cases not covered by this work. Namely, there are models with \(\tmin=-\infty\) (or \(\tmax=\infty\)) and \(\omin=-\infty\) (or \(\omax=\infty\), respectively). In these models, we could let \(\il=-\infty\) (or \(\ir=\infty\), respectively), and the formal proof will still work out. However, this would lead to an uncountable state space with unbounded densities (i.e., expected particle numbers per site). We conjecture that under suitable assumptions the construction of the dynamics can be established to validate the existence of the product stationary blocking measures \eqref{eq:mucdef}, but this is left for future work.

There are several models in the literature, see e.g., \(q\)-ASEP and KLS in Section \ref{sc:standup} with \(\omin=0\) and \(\omax=1\). Most of these are not covered directly by the above assumptions as the jump rates depend on more than the occupation of the departure and arrival sites of the jump. However, we introduce the \emph{stand-up transformation} in Section \ref{sc:standup}, which maps some of these models to one for which all our results apply. Hence the distributions \(\un\mu^c\) or \(\un\nu^n\) have direct relevance for such models as well.

One notable exemption is the ASEP. Both ASEP itself, and its stood-up version, the zero range process are fully covered by the above. As it turns out, comparing the unique stationary distributions that result gives a new proof of
\begin{tm}[Jacobi triple product (for some reals)]\label{tm:jacobi}
 Let \(0<X<1\) and \(Y\ne0\) be reals. Then
 \begin{equation}
  \prod_{i=1}^\infty\bigl(1-X^{2i}\bigr)\bigl(1+X^{2i-1}Y^2\bigr)\Bigl(1+\frac{X^{2i-1}}{Y^2}\Bigr)=\sum_{j=-\infty}^\infty X^{j^2}Y^{2j}.\label{eq:jacobi}
 \end{equation}
\end{tm}
We will prove this statement in Section \ref{sc:jacobi}, but note that it holds for any complex numbers \(X\), \(Y\) with \(|X|<1\) and \(Y\ne0\).

That Jacobi's triple product appears very naturally in the context of the two most classical interacting particle systems is somewhat surprising. This identity arises in various areas of mathematics, and has various proofs mostly using number theoretic arguments, see e.g., Wilf \cite{wilf_noth_triple_product}, Gasper and Rahman \cite{gasper_rahman_hypergeo}, or Andrews \cite{george_simple_triple_product}. It also appears in a combinatoric context, we refer to the survey of Pak \cite{pak_partitions}. Probabilistic arguments do not seem common in connection with this identity. Kemp \cite{kemp_q_bessel} uses it with some special statistical distributions, Ostrovsky \cite{ostrovsky_barnes} with Barnes distributions, and Ismail \cite{ismail_queuing} in connection with special queuing systems. Recently similar summation formulas have also appeared in exactly solvable particle systems (Corwin \cite{corwin_qhahn}, Borodin, Corwin and Sasamoto \cite{borodin_corwin_sasamoto_duality}).

The rest of this article introduces several examples in Section \ref{sc:cover} on which our results apply, and a few models in Section \ref{sc:standup} which are not directly covered but can be transformed to nevertheless enjoy the results. We then turn to proving the statements, first in general in Section \ref{sc:gen}, then concentrating on the finite \(N\) case in Section \ref{sc:jacobi}.

\section{Models we directly cover}\label{sc:cover}

We give several examples for which our results apply. We also refer to results in the literature, where available.

\subsection{Asymmetric simple exclusion}

The asymmetric simple exclusion (ASEP) is obtained by the choices \(\omin=0\), \(\omax=1\), \(s(y,\,z)=\mathbbm1\{y\le 1,\,z\le 1\}\), and \(f(y)=y\). Since \(I\) is finite, \(\tmin=-\infty\) and \(\tmax=\infty\), and product blocking measures exist on the whole of \(\Zb\). The marginals \eqref{eq:mutedef} become Bernoulli(\(\vr_i\)) with \(\vr_i=\frac{\e{\te_i}}{1+\e{\te_i}}=\frac{\e c(\frac pq)^i}{1+\e c(\frac pq)^i}\). Setting \(\te_i\) constant across the lattice results in translation-invariant non-reversible product stationary distributions.

Blocking measures in ASEP have been well known for a long time (Liggett \cite{couse}), and have been used (e.g., Ferrari, Kipnis and Saada \cite{fks}). When the nearest neighbor assumption is dropped, the picture becomes highly nontrivial, see e.g., Ferrari, Lebowitz and Speer \cite{fer_leb_speer_blocking}, Bramson and Mountford \cite{bram_mount_block}, and Bramson, Liggett and Mountford \cite{bram_lig_mountf_char_stati}. We do not consider this case here.

\subsection{Asymmetric \(K\)-exclusion}

This model is obtained by fixing a \(K\ge2\) integer, and generalizing the ASEP to \(\omin=0\) and \(\omax=K\), \(s(y,\,z)=\mathbbm1\{y\le K,\,z\le K\}\), \(f(y)=\mathbbm1\{y\ge1\}\). Again, \(\tmin=-\infty\) and \(\tmax=\infty\), and we obtain product reversible blocking measures on the whole of \(\Zb\), with truncated Geometric marginals. These have been known for \(K\)-exclusion before.

In contrast to the blocking situation, the product measure of marginals \eqref{eq:mutedef} is \emph{not} stationary if \(\te_i\) is kept constant across the lattice. The structure of translation-invariant stationary distributions is unknown, in fact even the existence of extremal stationary and translation-invariant measures for all densities between \(0\) and \(K\) is not established. Nevertheless, there are strong hydrodynamic results by Sepp\"al\"ainen \cite{hkl,seck}, and Bahadoran, Guiol, Ravishankar and Saada \cite{bagurasa,bagurasastrong}.

\subsection{Asymmetric rate one zero range process}\label{sc:1zr}

Zero range processes are obtained by the choices \(\omin=0\), \(\omax=\infty\), \(s(y,\,z)\equiv1\), and \(f\) any non-decreasing function with \(f(0)=0\), \(f(y)>0\) for \(y>0\). Construction (up to bounded increment \(f\)'s) and discussion of the stationary distributions can be found in Liggett \cite{lize} and Andjel \cite{and}, this was later partially extended for faster growing \(f\)'s by Bal\'azs, Sepp\"al\"ainen, Sethuraman and Rassoul-Agha \cite{exists}. We remark that in our countable state space \(\Om\) no restriction (other than attractivity) is needed on \(f\) to construct the zero range dynamics.

The most common choice for the rate function is \(f(y)=\mathbbm1\{y>0\}\), we refer to this as the rate one zero range process. The marginals \eqref{eq:mutedef} become Geometric with parameters \(\al_i=1-\e{\te_i}\). The bounds \eqref{eq:tminmaxdef} become \(\tmin=-\infty\), \(\tmax=0\), which forces \(\ir<\infty\) by Remark \ref{rm:finite}, while the volume \(\La\) can be kept half-infinite to the left.

\subsection{Asymmetric independent walkers}

This is a variant of zero range with \(f(y)=y\). Here particles jump independently of each other, and the marginals \eqref{eq:mutedef} become Poisson. As \(\tmin=-\infty\) and \(\tmax=\infty\), \eqref{eq:teidef} does not impose restrictions on \(\il\) or \(\ir\). However, as \(\omax=\infty\), our assumption \eqref{eq:Omdef} does not allow \(\ir=\infty\) as this would imply an uncountable state space, see the remark after Corollary \ref{cr:nun}.

\subsection{Asymmetric \(q\)-zero range process}\label{sc:qzr}

For later purposes we emphasize yet another special choice of zero range processes: \(f(y)=1-{\hat q}^y\) with a parameter \(0\le\hat q<1\). (This parameter has nothing to do with the asymmetry \(q=1-p\).) The totally asymmetric version \(p=1-q=1\) of this process was considered in Bal\'azs, Komj\'athy and Sepp\"al\"ainen \cite{unipq3}. We again have \(\ir<\infty\) by Remark \ref{rm:finite}.

\subsection{An asymmetric are-you-alone process}\label{sc:alone}

Again for later purposes we now consider a very particular choice. Let \(\omin=0\), \(\omax=\infty\), fix \(|\de|\le\ve<1\) parameters, and abbreviate by \(\ge z\) any integer at least \(z\) in the arguments below:
\[
 \begin{aligned}
  s(1,\,1):&=\frac{1-\de}{1-\ve},&\qquad s(1,\,\ge2)=s(\ge2,\,1):&=1,&\qquad s(\ge2,\,\ge2):&=\frac{1+\de}{1+\ve},\\
  f(0):&=0,&\qquad f(1):&=1-\ve,&\qquad f(\ge2):&=1+\ve.
 \end{aligned}
\]
These result in
\[
 \begin{aligned}
  \pf(1,\,0)&=p\cdot(1-\de),&\qquad \qf(0,\,1)&=q\cdot(1-\de),\\
  \pf(1,\,\ge1)&=p\cdot(1-\ve),&\qquad \qf(\ge1,\,1)&=q\cdot(1-\ve),\\
  \pf(\ge2,\,0)&=p\cdot(1+\ve),&\qquad \qf(0,\,\ge2)&=q\cdot(1+\ve),\\
  \pf(\ge2,\,\ge1)&=p\cdot(1+\de),&\qquad \qf(\ge1,\,\ge2)&=q\cdot(1+\de),
 \end{aligned}
\]
and zero in all other cases. The rates are only sensitive to no particles, a lonely particle, or at least two particles on sites, and the resulting marginals \eqref{eq:mutedef} are distorted Geometrics. For this model \(\tmin=-\infty\) but \(\tmax<\infty\) hence \(\ir<\infty\) is required. We remark that, besides Theorem \ref{tm:mucstati} showing the structure of product blocking measures, keeping \(\te_i\) constant across \(\Zb\) results in translation-invariant product stationary distributions.

\subsection{Asymmetric bricklayers}

We finish the line of examples on which our results apply directly by a natural model with \(\omin=-\infty\) and \(\omax=\infty\). Set any non-decreasing (and non-constant) function \(f\,:\,\Zb\to\Rb^+\) with the property that \(f(z)\cdot f(1-z)=1\) for all \(z\in\Zb\). Let \(s(y,\,z)=1+\frac1{f(y)f(z)}\). Then
\[
 \pf(y,\,z)=p\cdot\bigl(f(y)+f(-z)\bigr),\qquad\qf(y,\,z)=q\cdot\bigl(f(-y)+f(z)\bigr).
\]
A natural choice is \(f(z)=\e{\beta(z-\frac12)}\), in which case \(\tmin=-\infty\) and \(\tmax=\infty\). As explained after Corollary \ref{cr:nun}, we restrict both \(\il>-\infty\) and \(\ir<\infty\) for countability reasons, but conjecture that existence of the dynamics and product blocking measures on doubly infinite volumes \(\La\) should hold without such restrictions. Construction of this model in the translation-invariant case was also carried out in Bal\'azs, Sepp\"al\"ainen, Sethuraman and Rassoul-Agha \cite{exists}.

\section{Models we first stand up and then cover}\label{sc:standup}

We now explain how a simple combinatorial transformation extends our results to models which are not directly covered by the assumptions we made on the dynamics. This is also an important step in our probabilistic proof of the Jacobi triple product. In this section we consider models \(\un\om(t)\in\Om\) with \(\omin=0\), \(\omax=\infty\), \(\tmin=-\infty\), \(\tmax<\infty\), \(\il=-\infty\) and \(\ir=0\). We set
\begin{equation}
 \te_i=\tmax+(i-1)\cdot(\ln p-\ln q)\qquad(i\le0),\label{eq:aritei}
\end{equation}
and on the right boundary 
\[
 \pf_\ir(y)=\lim_{z\to\infty}\pf(y,\,z),\qquad\qf_\ir(y)=\lim_{z\to\infty}\qf(y,\,z).
\]
We also assume that these limits are non-zero for all \(y>0\), see the remark after Theorem \ref{tm:mucstati}. Notice that we have \(N_\text p(\un z)<\infty\) for \(\un z\in\Om\) under this setup.

Next we construct another particle system \(\un\eta(t)\) from \(\un\om(t)\) and a fixed integer \(n\). This is done by defining the lay-down operation \(\Lc^n\,:\,\Om\to\{0,\,1\}^\Zb;\ \un z\mapsto\un a\), where \(\un a\) is defined as follows. Set
\begin{equation}
 r_0(\un z)=n-N_\text p(\un z)+1,\qquad\text{and}\qquad r_{i+1}(\un z)=r_i+z_{-i}(\un z)+1\qquad(i\ge0)\label{eq:rdef}
\end{equation}
and
\[
 \bigl(\Lc^n(\un z)\bigr)_k=a_k=\left\{
  \begin{aligned}
   &1,&&\text{if }k=r_i(\un z)\text{ for some }i\ge0,\\
   &0,&&\text{otherwise.}
  \end{aligned}
 \right.
\]
In words, the configuration \(\Lc^n(\un z)\) has a leftmost particle at position \(n-N_\text p(\un z)+1\), and gap sizes equal to particle numbers on consecutive sites (from right to left) of \(\un z\) (\emph{laying \(\un z\) down}). Figure \ref{fig:ld} demonstrates \(\Lc^n(\un z)\) and how the original \(z_i\) variables appear, it is assumed that no particles are present left of what we see.

\begin{figure}[ht]
 \begin{center}
  \begin{pspicture}(330,30)
   \psline{->}(5,10)(325,10)
   \uput{2pt}[315](325,10){\small\(\Zb\)}
   \multido{\n=15+20}{16}{
    \psline(\n,8)(\n,12)
   }
   \rput(55,15){\(\bullet\)}
   \rput(115,15){\(\bullet\)}
   \rput(135,15){\(\bullet\)}
   \rput(155,15){\(\bullet\)}
   \rput(195,15){\(\bullet\)}
   \rput(215,15){\(\bullet\)}
   \rput(255,15){\(\bullet\)}
   \rput(275,15){\(\bullet\)}
   \rput(295,15){\(\bullet\)}
   \rput(315,15){\(\bullet\)}

   \rput[t](55,6){\tiny\(r_0\)}
   \rput[t](115,6){\tiny\(r_1\)}
   \rput[t](135,6){\tiny\(r_2\)}
   \rput[t](155,6){\tiny\(r_3\)}
   \rput[t](195,6){\tiny\(r_4\)}
   \rput[t](215,6){\tiny\(r_5\)}
   \rput[t](255,6){\tiny\(r_6\)}
   \rput[t](275,6){\tiny\(r_7\)}
   \rput[t](295,6){\tiny\(r_8\)}
   \rput[t](315,6){\tiny\(r_9\)}

   \rput[b](85,18){\tiny\(z_0\!=\!2\)}
   \rput[b](125,18){\tiny\(z_{-1}\!=\!0\)}
   \rput[b](175,18){\tiny\(z_{-3}\!=\!1\)}
   \rput[b](235,18){\tiny\(z_{-5}\!=\!1\)}
   \rput[b](305,18){\tiny\(z_{-8}\!=\!0\)}

   \psarc{<-}(165,-8){25}{75}{105}
  \end{pspicture}
  \caption{The configuration \(\Lc^n(\un z)\). The arrow indicates the move \(\un z\to\un z^{-3,-2}\) or equivalently \(r_3\to r_3+1\).}\label{fig:ld}
 \end{center}
\end{figure}

\begin{lm}\label{lm:standup}
 The function \(\Lc^n\) is actually a bijection from \(\Om\) to
 \begin{equation}
  H^n:\,=\{\un a\in\{0,\,1\}^\Zb\,:\,N(\un a)=n\}.\label{eq:hndef}
 \end{equation}
 Its inverse, the \emph{stand-up operation} \(\Uc\) is given by the following procedure. First, find the leftmost particle in \(\un a\in H^n\):
 \[
  R_0(\un a)=\min\{k\,:\,a_k=1\},
 \]
 which is finite by \(N(\un a)=n>-\infty\). Then set recursively
 \[
  R_{i+1}(\un a)=\min\{k>R_i(\un a)\,:\,a_k=1\},\qquad(i\ge0).
 \]
 Finally, let
 \[
  \bigl(\Uc(\un a)\bigr)_i=R_{1-i}(\un a)-R_{-i}(\un a)-1\qquad(i\le0).
 \]
\end{lm}
We postpone the proof to Section \ref{sc:jacobi}.

We now trace how individual moves of the process \(\un\om(t)\) happen in \(\un\eta(t):\,=\Lc^n(\un\om(t))\). A possible step \(\un\om\to\un\om^{0,1}\) can happen on the right boundary, which simply decreases \(\om_0\) by one. This step happens with rate \(\pf_\ir(\om_0)\), and the result is a decrease of \(N_\text p(\un\om)\) by one. This moves the leftmost particle of \(\un\eta\), \(r_0\) to the right by one while nothing else moves in \(\un\eta\). The step \(\un\om\to\un\om^{1,0}\) with rate \(\qf_\ir(\om_0)\) has the reverse effect. Other steps are of the form \(\un\om\to\un\om^{-i,-i+1}\) for \(i>0\), happening with rate \(\pf(\om_{-i},\,\om_{-i+1})\), or the reverse \(\un\om\to\un\om^{-i+1,-i}\), with rate \(\qf(\om_{-i},\,\om_{-i+1})\). These do not affect \(N_\text p(\un\om)\), and simply move the \(i^\text{th}\) particle of \(\un\eta\), \(r_i\) by one step to the right or left, respectively. One such step is indicated in Figure \ref{fig:ld}. This way an interacting particle system \(\un\eta(t)\) is constructed, and picking \(\un\om\) in its unique blocking measure, \(\un\eta\in H^n\) is also automatically in a stationary blocking distribution.

\subsection{Asymmetric simple exclusion}

If \(\un\om(t)\) is the rate one zero range process of Section \ref{sc:1zr}, then its laid-down version \(\un\eta(t)=\Lc^n(\un\om(t))\) is ASEP on \(\Zb\) with the same asymmetry parameters \(p=1-q\). This will enable us to proceed to the Jacobi triple product in Section \ref{sc:jacobi}.

\subsection{Asymmetric \(q\)-simple exclusion}

The laid-down of the asymmetric \(q\)-ZRP in Section \ref{sc:qzr} is the \(q\)-ASEP, hence a blocking measure follows from our construction on this process. In \(q\)-ASEP jump rates of particles depend on the distance to the nearest particle in the direction of the jump closing to 1 from below in an exponential manner with base \(\hat q\). The totally asymmetric version appeared in Borodin and Corwin \cite{bo_co_mcdonald}.

\subsection{Asymmetric Katz-Lebowitz-Spohn process}

The laid-down of the are-you-alone process of Section \ref{sc:alone} is the Katz-Lebowitz-Spohn process in one dimension \cite{kls}. Here particles repel each other when they are nearest neighbors i.e., jumping from a neighboring particle to an empty site with no neighbors happens with rates \(p(1+\ve)\) and \(q(1+\ve)\), while jumping from one with no neighbors to an empty site with a neighbor occurs with rates \(p(1-\ve)\) and \(q(1-\ve)\). The parameter \(\de\) can also tune the rates when the jump happens between two sites with no neighbors (\(p(1+\de)\) and \(q(1+\de)\)) or between two sites with neighbors (\(p(1-\de)\) and \(q(1-\de)\)). Blocking measures of independent inter-particle distances follow. See Zia \cite{zia_kls} for a review on this process. The correspondence with the are-you-alone process appeared in R\'akos \cite{rakos_kls_misa}.

\section{General proofs}\label{sc:gen}

We start with showing that under the attractivity conditions the dynamics is well-defined.
\begin{proof}[Proof of Lemma \ref{lm:noblowup}]
 We emphasize again that by construction our state space is countable. Hence we can directly apply the theory of countable Markov chains without the need of using heavy analytic tools like semigroup and generator machinery. In particular, we do not need to look at Feller property; our processes indeed might fail to be Feller (using the discrete topology of the countable state space) with attractivity being the only assumption on the jump rates. Nevertheless the infinitesimal description gives a unique definition of the Markov process as soon as one demonstrates that the rates are non-explosive, see \cite[Chapter 2.5]{liggett_cont_markov}. Hence we turn to proving this using the attractivity assumption only.
 \begin{itemize}
  \item When both \(\omax\) and \(\omin\) are finite, then the jump rates are uniformly bounded and no explosion can occur.
  \item When both are infinite, by \eqref{eq:Omdef} and \eqref{eq:NpNh} both \(\il\) and \(\ir\) are finite. If \(\ir=0=\il\), then our process is a two-sided birth and death process with decreasing birth rates towards positive values and decreasing death rates towards negative values, and no explosion can occur. Otherwise, without loss of generality, we assume \(\ir>0\), and for the process \(\un\om(t)\) evolving according to the above dynamics, define the \emph{height function} initially by
  \[
   h_{\frac12}(0):\,=0;\qquad h_{k+\frac12}(0):\,=\left\{
    \begin{aligned}
     &h_{\frac12}(0)+\sum_{i=k+1}^0\om_i(0),&&\text{for }\il-1\le k<0,\\
     &h_{\frac12}(0)-\sum_{i=1}^k\om_i(0),&&\text{for }0<k\le\ir.
    \end{aligned}
   \right.
  \]
  Increasing \(h_{k+\frac12}\) by one for a \(\un\om\to{\un\om}^{k,k+1}\) jump and decreasing it for a \(\un\om\to{\un\om}^{k+1,k}\) jump (including the boundary jumps!) will then keep the above display for all later times \(t\). Now notice that \(\max_{\il-1\le k\le\ir}h_{k+\frac12}(t)\) increases by rates bounded by \((\ir-\il)\cdot\pf(0,\,0)+\qf_\il(0)+\pf_\ir(0)\) and similarly, \(\min_{\il\le k<\ir}h_{k+\frac12}(t)\) decreases by rates bounded by \((\ir-\il)\cdot\qf(0,\,0)+\pf_\il(0)+\qf_\ir(0)\), again no explosion can occur.
 \item When \(\omax=\infty\) but \(\omin\) is finite, then \(\ir<\infty\), and the total number of particles, \(\sum_{i=\il}^\ir(\om_i(t)-\omin)\) starts with a finite value at time zero. It can only increase at the boundary (or boundaries, if \(\il>-\infty\)) with rate at most \(\qf_\ir(\om_\ir(t))\le\qf_\ir(\omin)\) (or also \(\pf_\il(\om_\il(t))\le\pf_\il(\omin)\), if \(\il>-\infty\)), and again no explosion can occur. The case \(\omin=-\infty\) and \(\omax\) finite is handled similarly.
 \end{itemize}
\end{proof}

Next we prove the stationarity result for the distribution \(\un\mu^c\).
\begin{proof}[Proof of Theorem \ref{tm:mucstati}]
 The statement has two main parts. First, the measure \(\un\mu^c\) is concentrated on \(\Om\) and second, it is stationary and reversible for our dynamics.

 For the first part, the interesting case is when \(\il=-\infty\) or \(\ir=\infty\). In these cases \(\omin\) or \(\omax\), respectively, are finite and we need to show that \(N_\text p\) or \(N_\text h\), respectively, are \(\un\mu^c\)-a.s.\ finite. For finite \(\omin\) or \(\omax\) the marginal \eqref{eq:mutedef} can be rewritten into, respectively,
 \begin{equation}
  \mu^\te(z)=\frac{\e{\te(z-\omin)}/f(z)!}{\sum\limits_{y\in I}\e{\te(y-\omin)}/f(y)!}=\frac{\e{-\te(\omax-z)}/f(z)!}{\sum\limits_{y\in I}\e{-\te(\omax-y)}/f(y)!}.\label{eq:muways}
 \end{equation}
 We show that for \(\il=-\infty\) and \(\omin>-\infty\), \(N_\text p\) is \(\un\mu^c\)-a.s.\ finite, the case \(\ir=\infty\) and \(\omax<\infty\) is very similar. Define \(A_i\), \(i\le0\), as the event that there is a particle at position \(i\):
 \[
  A_i:\,=\{\un z\,:\,z_i\ne\omin\}.
 \]
 Then
 \[
  \begin{aligned}
   \un\mu^c\{A_i\}&=\frac{\sum\limits_{z=\omin+1}^{\omax}\e{\te_i(z-\omin)}/f(z)!}{\sum\limits_{y\in I}\e{\te_i(y-\omin)}/f(y)!}=\e{\te_i}\frac{\sum\limits_{x=\omin}^{\omax-1}\e{\te_i(x-\omin)}/f(x+1)!}{\sum\limits_{y\in I}\e{\te_i(y-\omin)}/f(y)!}\\
   &\le\frac{\e{\te_i}}{f(\omin)}\frac{\sum\limits_{x=\omin}^{\omax-1}\e{\te_i(x-\omin)}/f(x)!}{\sum\limits_{y\in I}\e{\te_i(y-\omin)}/f(y)!}\le\frac{\e{\te_i}}{f(\omin)}.
  \end{aligned}
 \]
 With the choices \eqref{eq:teidef} this is summable for \(i\le0\), hence Borel-Cantelli ensures \(\un\mu^c\)-a.s.\ finitely many occurrence of the \(A_i\)'s which implies \(\un\mu^c\)-a.s.\ finiteness of \(N_\text p\).

 We now turn to proving reversibility of \(\un\mu^c\) w.r.t.\ the dynamics (notice that this implies stationarity as well for any Markov chain). First notice that \(\un\mu^c(\un z)\ne0\) for any \(\un z\in\Om\). The generators \(L^\il\) (when \(\il>-\infty\)) and \(L^\ir\) (when \(\ir<\infty\)), and each summand in \(L^\text{bulk}\) describe disjoint moves and their inverses, and reversibility follows from
 \[
  \begin{aligned}
   \un\mu^c(\un z)\pf_\il(z_\il)&=\un\mu^c(\un z^{\il-1,\il})\qf_\il(z_\il+1)\qquad&&\text{if }\il>-\infty,\\
   \un\mu^c(\un z)\pf_\ir(z_\ir)&=\un\mu^c(\un z^{\ir,\ir+1})\qf_\ir(z_\ir-1)\qquad&&\text{if }\ir<\infty,\\
   \un\mu^c(\un z)\pf(z_i,\,z_{i+1})&=\un\mu^c(\un z^{i,i+1})\qf(z_i-1,\,z_{i+1}+1),\qquad&&\il-1<i<\ir.
  \end{aligned}
 \]
 Expanding these via the definitions, and simplifying the product measure on all unchanged bits of \(\un z\) gives
 \[
  \begin{aligned}
   \frac1{Z(\te_\il)}\frac{\e{\te_\il z_\il}}{f(z_\il)!}\pf_\il(z_\il)&=\frac1{Z(\te_\il)}\frac{\e{\te_\il(z_\il+1)}}{f(z_\il+1)!}\qf_\il(z_\il+1)\\
   \frac1{Z(\te_\il)}\frac{\e{\te_\ir z_\ir}}{f(z_\ir)!}\pf_\il(z_\ir)&=\frac1{Z(\te_\ir)}\frac{\e{\te_\ir(z_\ir-1)}}{f(z_\ir-1)!}\qf_\ir(z_\ir-1)\\
   \frac{\e{\te_iz_i}\e{\te_{i+1}z_{i+1}}ps(z_i,\,z_{i+1}+1)f(z_i)}{Z(\te_i)Z(\te_{i+1})f(z_i)!f(z_{i+1})!}&=\frac{\e{\te_i(z_i-1)}\e{\te_{i+1}(z_{i+1}+1)}qs(z_i,\,z_{i+1}+1)f(z_{i+1}+1)}{Z(\te_i)Z(\te_{i+1})f(z_i-1)!f(z_{i+1}+1)!}
  \end{aligned}
 \]
 which in turn directly follow from \eqref{eq:teidef}.
\end{proof}

We proceed with investigating the doubly infinite volume case and the conditional distributions \(\un\nu^n\).
\begin{proof}[Proof of Lemma \ref{lm:nunwd}]
 That the conditional distribution is well defined follows from the fact that \(\un\mu^c\) is positive for all states in \(\Om\), hence the condition is non-degenerate. Using the two forms \eqref{eq:muways} of the marginals, and \eqref{eq:teidef}, any state \(\un z\) with \(N(\un z)=n\) has
 \begin{equation}
  \un\nu^n(\un z)=\frac{\bigl(\prod\limits_{i\le0}\frac{\e{(c+i(\ln p-\ln q))(z_i-\omin)}}{f(z_i)!}\bigr)\bigl(\prod\limits_{i>0}\frac{\e{-(c+i(\ln p-\ln q))(\omax-z_i)}}{f(z_i)!}\bigr)}{\sum\limits_{\un y\,:\,N(\un y)=n}\bigl(\prod\limits_{i\le0}\frac{\e{(c+i(\ln p-\ln q))(y_i-\omin)}}{f(y_i)!}\bigr)\bigl(\prod\limits_{i>0}\frac{\e{-(c+i(\ln p-\ln q))(\omax-y_i)}}{f(y_i)!}\bigr)}.\label{eq:nunexpanded}
 \end{equation}
 Notice that the products are finite, hence in this form the denominators in \eqref{eq:muways} are also finite and could already be cancelled. Next, separating the factors with \(c\) in them gives \(\e{-cN(\un z)}\) in the numerator, and \(\e{-cN(\un y)}\) in each term of the sum in the denominator. Since both are \(\e{-cn}\), they cancel out from \(\un\nu^n(\un z)\).
\end{proof}

We now start preparing the proof of Lemma \ref{lm:standup} while keeping general finite \(\omin\) and \(\omax\) for the rest of this section. Define the shift by an integer \(j\) as
\[
 (\tau^j\un z)_i:\,=z_{i+j},\qquad\text{and abbreviate}\qquad\tau:\,=\tau^1.
\]
Then a simple calculation gives
\begin{equation}
 N(\tau\un z)=\sum_{i=1}^\infty(\omax-z_{i+1})-\sum_{i=-\infty}^0(z_{i+1}-\omin)=N(\un z)-(\omax-\omin),\label{eq:ntau}
\end{equation}
and recursively
\begin{equation}
 N(\tau^j\un z)=N(\un z)-j\cdot(\omax-\omin).\label{eq:nshift}
\end{equation}

To prepare the proof of the Jacobi triple product, we investigate how \(\un\mu^c\) reacts to shifts, still in the doubly infinite volume case.
\begin{lm}\label{lm:muctau}
 For any \(i\in\Zb\) and \(\un z\in\Om\),
 \begin{equation}
  \un\mu^c(\tau^j\un z)=\Bigl(\frac qp\Bigr)^{(\omax-\omin)\frac{j^2-j}2-N(\un z)j}\cdot\e{c(\omax-\omin)j}\cdot\un\mu^c(\un z).\label{eq:muctau}
 \end{equation}
\end{lm}
\begin{proof}
 The starting point is the finite-product expansion, as seen in \eqref{eq:nunexpanded}, of \(\frac{\un\mu^c(\tau\un z)}{\un\mu^c(\un z)}\), where again the normalizations could be cancelled out in this way of writing the products:
 \[
  \begin{aligned}
   \frac{\un\mu^c(\tau\un z)}{\un\mu^c(\un z)}&=\frac{\bigl(\prod\limits_{i\le0}\frac{\e{(c+i(\ln p-\ln q))(z_{i+1}-\omin)}}{f(z_{i+1})!}\bigr)\bigl(\prod\limits_{i>0}\frac{\e{-(c+i(\ln p-\ln q))(\omax-z_{i+1})}}{f(z_{i+1})!}\bigr)}{\bigl(\prod\limits_{i\le0}\frac{\e{(c+i(\ln p-\ln q))(z_i-\omin)}}{f(z_i)!}\bigr)\bigl(\prod\limits_{i>0}\frac{\e{-(c+i(\ln p-\ln q))(\omax-z_i)}}{f(z_i)!}\bigr)}\\
   &=\frac{\bigl(\prod\limits_{j\le0}\frac{\e{(c+(j-1)(\ln p-\ln q))(z_j-\omin)}}{f(z_j)!}\bigr)\bigl(\prod\limits_{j>0}\frac{\e{-(c+(j-1)(\ln p-\ln q))(\omax-z_j)}}{f(z_j)!}\bigr)}{\bigl(\prod\limits_{i\le0}\frac{\e{(c+i(\ln p-\ln q))(z_i-\omin)}}{f(z_i)!}\bigr)\bigl(\prod\limits_{i>0}\frac{\e{-(c+i(\ln p-\ln q))(\omax-z_i)}}{f(z_i)!}\bigr)}\\
   &\quad\cdot\frac{\e{c(z_1-\omin)}}{f(z_1)!}\cdot\frac{f(z_1)!}{\e{-c(\omax-z_1)}}\\
   &=\e{(\ln p-\ln q)N(\un z)+c(\omax-\omin)}=\Bigl(\frac pq\Bigr)^{N(\un z)}\e{c(\omax-\omin)}.
 \end{aligned}
 \]
 Applying this now for \(\tau^j\un z\) in combination with \eqref{eq:nshift} gives
 \[
  \un\mu^c(\tau^{j+1}\un z)=\Bigl(\frac pq\Bigr)^{N(\un z)-j(\omax-\omin)}\e{c(\omax-\omin)}\cdot\un\mu^c(\tau^j\un z),
 \]
 the solution of which, with initial data \(\un\mu^c(\tau^0\un z)=\un\mu^c(\un z)\), is \eqref{eq:muctau}.
\end{proof}
From now on, Greek quantities will denote random variables distributed according to the measures they are featured in. Formula \eqref{eq:muctau} gives partial information on the distribution of \(N(\un\om)\) under the measure \(\un\mu^c\).
\begin{cor}\label{cr:mucN}
 For any \(j,\ n\in\Zb\),
 \[
  \un\mu^c\bigl\{N(\un\om)=n-j(\omax-\omin)\bigr\}=\Bigl(\frac qp\Bigr)^{(\omax-\omin)\frac{j^2-j}2-nj}\e{c(\omax-\omin)j}\un\mu^c\{N(\un\om)=n\}.
 \]
\end{cor}
\begin{proof}
 Notice \(\tau^j\) is one-to-one, and recall \eqref{eq:nshift}.
 \begin{multline*}
  \un\mu^c\{N(\un\om)=n-j(\omax-\omin)\}=\sum_{\un z\,:\,N(\un z)=n-j(\omax-\omin)}\un\mu^c(\un z)\\
  \begin{aligned}
   &=\sum_{\un y\,:\,N(\un y)=n}\un\mu^c(\tau^j\un y)\\
   &=\Bigl(\frac qp\Bigr)^{(\omax-\omin)\frac{j^2-j}2-nj}\e{c(\omax-\omin)j}\sum_{\un y\,:\,N(\un y)=n}\un\mu^c(\un y)\\
   &=\Bigl(\frac qp\Bigr)^{(\omax-\omin)\frac{j^2-j}2-nj}\e{c(\omax-\omin)j}\un\mu^c\{N(\un\om)=n\}.
  \end{aligned}
 \end{multline*}
\end{proof}
We can now also see how the conditional distributions \(\un\nu^n\) react to shifts.
\begin{cor}\label{cr:nuinv}
 For any \(n,\ j\in\Zb\), and \(\un z\in\Om\) with \(N(\un z)=n\),
 \[
  \un\nu^{n-j(\omax-\omin)}(\tau^j\un z)=\un\nu^n(\un z).
 \]
\end{cor}
\begin{proof}
 Just apply the definitions and the above.
 \[
  \begin{aligned}
   \un\nu^{n-j(\omax-\omin)}(\tau^j\un z)&=\frac{\un\mu^c(\tau^j\un z)}{\un\mu^c\bigl\{N(\un\om)=n-j(\omax-\omin)\bigr\}}\\
   &=\frac{\un\mu^c(\un z)}{\un\mu^c\bigl\{N(\un\om)=n\bigr\}}=\un\nu^n(\un z).
  \end{aligned}
 \]
\end{proof}

\section{ASEP, ZRP, and the Jacobi triple product}\label{sc:jacobi}

In this section we assume the setup of Section \ref{sc:standup} for the state space \(\Om\) and the model \(\un\om(t)\). Recall the definition \eqref{eq:hndef}, and notice that for \(\un a\in H^n\) and due to \(\emax-\emin=1-0=1\) (these now play the role of \(\omin\) and \(\omax\) in the ASEP \(\un\eta(t)\)), \eqref{eq:nshift} allows to fine tune \(N(\un a)\) in steps of 1 by simply shifting the configuration. Also, for the same reason, Corollaries \ref{cr:mucN} and \ref{cr:nuinv} allow recursions of the distribution of \(N(\un\eta)\) w.r.t.\ \(\un\mu^c\) and of \(\un\nu^n\) in steps of 1. This would not work in cases with \(\emax-\emin>1\).

\begin{proof}[Proof of Lemma \ref{lm:standup}]
 Pick \(\un z\in\Om\) as described in Section \ref{sc:standup}. From \(N_\text p(\un z)<\infty\) it follows that \(\Lc^n(\un z)\) has a rightmost hole, on the right of which all sites are occupied by a particle. In other words, both \(N_\text p\bigl(\Lc^n(\un z)\bigr)\) and \(N_\text h\bigl(\Lc^n(\un z)\bigr)\) are finite, and we can write
 \[
  N\bigl(\Lc^n(\un z)\bigr)=N\bigl(\tau^{r_0(\un z)-1}\Lc^n(\un z)\bigr)+r_0(\un z)-1=\sum_{i=-\infty}^0z_i+n-N_\text p(\un z)=n
 \]
 by \eqref{eq:nshift} and by construction of \(\Lc^n\). This shows \(\Lc^n(\un z)\in H^n\).

 Next pick any \(\un a\in H^n\). Not only \(\un a\) has a leftmost particle, but it also has a rightmost hole which shows \(\bigl(\Uc(\un a)\bigr)_i\ne0\) for only a finite number of indices \(i\le0\), in other words \(N_\text p\bigl(\Uc(\un a)\bigr)<\infty\) and \(\Uc(\un a)\in\Om\).
 
 Finally, \((\Uc\circ\Lc^n)(\un z)=\un z\) for all \(\un z\in\Om\) follows from the definition, while \((\Lc^n\circ\Uc)(\un a)=\un a\) for all \(\un a\in H^n\) comes from the fact that \(\tau\) changes \(N(\un a)\) by one and, given inter-particle distances for \((\Lc^n\circ\Uc)(\un a)\), the choice \(r_0\bigl(\Uc(\un a)\bigr)=n-N_\text p\bigl(\Uc(\un a)\bigr)+1\) is the only one among possible translations that results in \(N\bigl((\Lc^n\circ\Uc)(\un a)\bigr)=n\) and thus \((\Lc^n\circ\Uc)(\un a)\in H^n\).
\end{proof}

We now fully restrict our attention to the case of the rate 1 ZRP of Section \ref{sc:1zr} for \(\un\om(t)\), or equivalently ASEP for \(\un\eta(t)=\Lc^n(\un\om(t))\). The stationary blocking measures will be denoted by \(\un\mu\) for ZRP on \(\Om\) (no constant \(c\) here since the \(\te_i\)'s are fixed by \eqref{eq:aritei}), and by \(\un\pi^c\) on \(H\) and \(\un\nu^n\) on \(H^n\) for ASEP. The proof of Theorem \ref{tm:jacobi} will follow from the ergodic decomposition of \(\un\pi^c\) into its components \(\un\nu^n\), which we can fully work out due to \(\emax=1\) and \(\emin=0\) in \(H\). Then \(\un\mu\) and \(\un\pi^c\) can be compared via the functions \(\Lc^n\) and \(\Uc\).

Let
\[
 K^c:\,=\sum_{j=-\infty}^\infty\Bigl(\frac qp\Bigr)^{\frac{j^2+j}2}\e{-cj}.
\]
For any \(n\in\Zb\), Corollary \ref{cr:mucN} gives rise to the discrete Gaussian distribution
\[
 \un\pi^c\bigl\{N(\un\eta)=n\}=\frac1{K^c}\Bigl(\frac qp\Bigr)^{\frac{n^2+n}2}\e{-cn}
\]
by normalization. Also, \eqref{eq:ntau} and Corollary \ref{cr:nuinv} tell us that \(H\) decomposes into the disjoint union of the irreducible components \(H^n\), and \(\tau\) is the \(\un\nu^n\)-preserving bijection between \(H^n\)'s of consecutive indices. By definition of \(\un\nu^n\), the measure \(\un\pi^c\) has ergodic decomposition
\[
 \un\pi^c=\sum_{n=-\infty}^\infty\un\pi^c(\cdot\,|\,N(\un\eta)=n)\cdot\un\pi^c\bigl\{N(\un\eta)=n\}=\sum_{n=-\infty}^\infty\un\nu^n\cdot\frac1{K^c}\Bigl(\frac qp\Bigr)^{\frac{n^2+n}2}\e{-cn}.
\]
\begin{proof}[Proof of Theorem \ref{tm:jacobi}]
 Fix \(\un z\in\Om\), and \(n\in\Zb\). As discussed in Section \ref{sc:standup}, the one-to-one map \(\Lc^n\,:\,\Om\to H^n\) and its inverse \(\Uc\,:\,H^n\to\Om\) transfer the dynamics of ZRP into that of ASEP and back:
 \[
  \Bigl(\Lc^n\bigl(\un\om(t)\bigr)\,|\,\un\om(0)=\un z\Bigr)\overset{\text d}=\Bigl(\un\eta(t)\,|\,\un\eta(0)=\Lc^n(\un z)\Bigr).
 \]
 As both ASEP and ZRP have unique stationary distributions \(\un\nu^n\) and \(\un\mu\) by Corollary \ref{cr:nun} and the remark after Theorem \ref{tm:mucstati}, it follows that the random variables \(\un\eta\) and \(\un\om\) with these respective distributions satisfy \(\Lc^n(\un\om)\overset{\text d}=\un\eta\). Therefore,
 \begin{equation}
  \un\mu(\un z)=\un\nu^n\bigl(\Lc^n(\un z)\bigr)=\frac{\un\pi^c\bigl(\Lc^n(\un z)\bigr)}{\un\pi^c\bigl\{N(\un\eta)=n\}}=K^c\Bigl(\frac pq\Bigr)^{\frac{n^2+n}2}\e{cn}\cdot\un\pi^c\bigl(\Lc^n(\un z)\bigr).\label{eq:meq}
 \end{equation}
 for any \(\un z\in\Om\).

 We now substitute everything for our specific case of ZRP for \(\un\mu\). Marginals \eqref{eq:mutedef} with parameters \eqref{eq:aritei} and \(\tmax=0\) become Geometric(\(1-(\frac pq)^{i-1}\)) for sites \(i\le0\). Similarly, for \(\un\pi^c\) the marginals are Bernoulli(\(\frac{\e c(\frac pq)^k}{1+\e c(\frac pq)^k}\)) for sites \(k\in\Zb\). This parameter can also be written as \(\frac{\e c}{(\frac qp)^k+\e c}\). For simplicity we choose
 \[
  n=N_\text p(\un z)=\sum_{i=-\infty}^0z_i,
 \]
 other choices would not lead to novelty compared to the calculation seen in Lemma \ref{lm:muctau}. This results in \(r_0(\un z)=1\) in \eqref{eq:rdef}. Expanding \eqref{eq:meq} then gives
 \begin{multline}
  \Bigl\{\prod_{i=-\infty}^0\Bigl(\frac pq\Bigr)^{(i-1)z_i}\Bigl(1-\Bigl(\frac pq\Bigr)^{i-1}\Bigr)\Bigr\}\\
  =K^c\Bigl(\frac pq\Bigr)^{\frac{n^2+n}2}\e{cn}
  \Bigl\{\prod_{k=-\infty}^0\frac1{1+\e c(\frac pq)^k}\Bigr\}
  \prod_{i=-\infty}^0\Bigl\{\frac{\e c}{(\frac qp)^{r_{-i}}+\e c}\prod_{k=r_{-i}+1}^{r_{1-i}-1}\frac{(\frac qp)^k}{(\frac qp)^k+\e c}\Bigr\}\label{eq:combi}
 \end{multline}
 with \(r_{-i}=r_{-i}(\un z)\) of \eqref{eq:rdef}.

 We now consider the special case \(z_i=0\) for each \(i\le0\). Then \(n=0\) and \(r_{-i}=1-i\), and the above becomes, with some changes of signs of indices,
 \[
  \prod_{i=0}^\infty\Bigl(1-\Bigl(\frac qp\Bigr)^{i+1}\Bigr)=K^c\Bigl\{\prod_{k=0}^\infty\frac1{1+\e c(\frac qp)^k}\Bigr\}\prod_{i=0}^\infty\frac{\e c}{(\frac qp)^{i+1}+\e c}.
 \]
 We now re-arrange the products on the right hand-side into the left hand-side and conclude
 \[
  \prod_{i=1}^\infty\Bigl(1-\Bigl(\frac qp\Bigr)^i\Bigr)\Bigl(1+\e c\Bigl(\frac qp\Bigr)^{i-1}\Bigr)\Bigl(1+\e{-c}\Bigl(\frac qp\Bigr)^i\Bigr)=K^c=\sum_{j=-\infty}^\infty\Bigl(\frac qp\Bigr)^{\frac{j^2+j}2}\e{-cj}.
 \]
 Substitute \(X=\sqrt\frac qp\in(0,\,1)\) and \(Y^2=\e{-c}\bigl(\frac qp\bigr)^{\frac12}>0\) to obtain \eqref{eq:jacobi}.
\end{proof}
We remark that many of the formulas in this proof, as well as the Jacobi triple product can naturally be reformulated in terms of the \(q\)-Pochhammer symbols or \(q\)-shifted factorials, see e.g., Gasper and Rahman \cite{gasper_rahman_hypergeo}, or Borodin and Corwin \cite{bo_co_mcdonald} and references therein for definitions and identities of \(q\)-deformations of classical functions.

It seems tempting to try other configurations \(\un z\) in \eqref{eq:combi}. Any fixed value do not seem to add much novelty to our results. A more interesting attempt is to fix a marginal say, \(z_0\) only, and sum out all the other variables. This would make the left hand-side particularly simple. However, on the right hand-side either fixing \(n\), or fixing \(r_0\) seems a complicated issue. In the first case, the positions \(r_i\) depend on \(N_\text p(\un z)\), whereas in the second case \(n\) will depend on the same quantity. Thus, instead of reducing to a simple marginal, the right-hand side seems to require detailed information on the variables \(z_i\) that we wish to sum out.

\section*{Acknowledgments}
The authors thank anonymous referees for their helpful comments on a first version of this manuscript.

\bibliographystyle{plain}
\bibliography{refsmarton}
\end{document}